\documentclass[11pt]{amsart}
\usepackage[utf8]{inputenc} 
\usepackage{amsmath}
\usepackage{amsthm}
\usepackage{amsfonts}
\usepackage{natbib}
\usepackage{amssymb}
\usepackage{enumerate}
\usepackage{graphicx}
\theoremstyle{plain}
\newtheorem{thm}{Theorem}[section]

\newtheorem{cor}[thm]{Corollary}
\newtheorem{question}[thm]{Question}

\newtheorem{lemma}[thm]{Lemma}
\theoremstyle{definition}
\newtheorem{dfn}[thm]{Definition}
\theoremstyle{remark}
\newtheorem*{rmk}{Remark}

\newtheorem{example}[thm]{Example}
\newtheorem*{acknowledgments}{Acknowledgments}
\usepackage{mathtools}
\usepackage{mathbbol}
\usepackage{wasysym}
\usepackage{amssymb}
\usepackage{mathcomp}
\usepackage{MnSymbol}
\usepackage{comment}
\usepackage{enumerate}
\usepackage{ushort}

\usepackage{tabularx,booktabs,tikz}
\usepackage{caption}
\usepackage{amsmath}
\usepackage{amsfonts}

\usepackage{amscd}
\usepackage{amsthm}
\usepackage{amssymb} \usepackage{latexsym}
\usepackage{amscd}
\usepackage{amsthm}
\usepackage{amssymb} \usepackage{latexsym}
\usepackage{eufrak}
\usepackage{euscript}
\usepackage{epsfig}
\usepackage{graphics}
\usepackage{array}
\usepackage{enumerate}
\usepackage{dsfont}
\usepackage{color}
\usepackage{wasysym}
\usepackage{hyperref}
\usepackage{pdfsync}

\newcommand{\bel}[1]{\begin{equation}\label{#1}}

\newcommand{\be}{\begin{equation}}

\begin{document}
\title[Actions of some groups on the product of Hadamard spaces]{Remarks on the actions of some groups on the product of Hadamard spaces}
\author{Zunwu He}
\address{Zunwu He:School of Mathematical Sciences, Fudan University, Shanghai 200433, China}
\email{hzw@fudan.edu.cn}


\maketitle
\begin{abstract}
For a product of Hadamard spaces $X=X_1\times X_2$ on which some group $G\subset Is(X_1)\times Is(X_2)$ acting, G. Link \cite{1} introduced the growth rate $\delta_{\theta}$ of slope $\theta$ to construct a $G-$invariant $(b,\theta)$-density.

First, we show that $\delta_{\theta}$ is continuous in the slope $\theta$ as the above group action with some mild condition.

Second, we give a negative answer to a question raised by G. Link in \cite{1} in general. And further results about the question are discussed.

In the end, we can extent the $(b,\theta)$-densities in some reasonable condition.
\end{abstract}

\section{Introduction}
Recall that a \emph{Hadamard space} means a complete simply connected metric space with non-positive Alexandrov curvature. Let $X=X_1\times X_2$ be a product of proper
Hadamard spaces with the standard product metric, one can compactify $X$, $X_1$ and $X_2$ by giving the corresponding canonical geometric boundaries $\partial X$, $\partial X_1$
and $\partial X_2$ (\cite{7}, \cite{9}).
A basic fact is that $\partial_{reg} X=\partial X_1\times \partial X_2\times (0,\pi/2)$, $\partial_{sing}X=\partial X_1\bigcup \partial X_2$ and $\partial X=\partial_{reg} X\cup \partial_{sing} X$ (\cite{2}). The third component in the first term is called the slope of a point in $\partial X_{reg}$. In particular, $\partial X_1$ is of slope $0$ and $\partial X_2$ is of slope $\pi/2$.

Consider a group $G\subset Is(X_1)\times Is(X_2)$ acting properly on the space $X$, and $G$ contains a pair of isometries which projects a pair of independent rank one elements in each factor and satisfies \emph{the regular growth condition}(see the definition in the next section).

In order to investigate the limit set of Fuchsian groups acting on the hyperbolic plane, Patterson used the associated Poincar{\'e} series to construct a class of remarkable conformal measures and densities (\cite{3}). Sullivan made a profound generalization of Patterson's results (\cite{5}). He established the conformal measures and densities theory in the view of analyses, geometries and dynamics for discrete isometrical groups of the classical hyperbolic spaces, and related the growth rate to the Hausdorff dimension of the limit set. Coornaert developed the Patterson-Sullivan theory for the proper actions on Gromov hyperbolic spaces (\cite{Coornaert;1993}). Dal'bo et al. constructed the Patterson-Sullivan theory for Hadamard manifolds (\cite{Dal'bo;2000}).  Yang generalized the conformal measures and densities to the relatively hyperbolic groups by the so-called Partial Shadow Lemma (\cite{10}).

However, it is not efficacious to detect the geometry of limit sets for the higher rank symmetric spaces and Euclidean buildings by the generalized Patterson-Sullivan measures and densities (\cite{Albuquerque:1999}). Aimed to overcome this difficulty, Quint and Link independently introduced a class of generalized conformal densities and measures on each invariant subsets of the limit sets (\cite{4}, \cite{Link;2004}, \cite{2}, \cite{1}, \cite{15}, \cite{14}). Related problems were studied by Marc for graphs of convex cocompact graphs in a product of rank one symmetric spaces (\cite{Marc:1993}), and by Dal'bo et al. for discrete isometrical groups of a product of pinched negatively curved Hadamard manifolds (\cite{Dal'Bo_Kim:2008}).

Link introduced the growth rate of $G$ with some slope, which is quite important to construct the generalized measures and densities (\cite{2}, \cite{1}, \cite{15}, \cite{14}). She also showed that the growth rate is semi-upper continuous in the slope. We can further improve this result under some mild assumptions.

It is formulated as follows.
\begin{thm}\label{main theorem1}
The growth rate $\delta_{\theta}$ of $G$ with slope $\theta$ is continuous in $\theta\in [0,\pi/2]$, where $G$ is as above.
\end{thm}
By Theorem \ref{main theorem1}, we can get some important corollaries in the sequel.

\begin{cor}
There exists some $\theta_{\star}\in [0,\pi/2]$ such that
$$\delta_{\theta_{\star}}=\max\limits_{\theta\in[0,\pi/2]}\delta_{\theta}.$$
\end{cor}

\begin{rmk}
In \cite{1}, G. Link showed that such $\theta_\star$ is unique with $\delta(G)= \delta_{\theta_\star}$, where $G$ is as above which does not necessarily satisfy the regular growth condition.
\end{rmk}

Combining some extra work with Theorem \ref{main theorem1}, one can propose some interesting results.




First, we can get an interesting result on $\delta_\theta$.
\begin{thm}\label{pos exp}
It holds that the growth rate $\delta_\theta$ of the action $G\curvearrowright X=X_1\times X_2$ with slope $\theta$ is positive for any $\theta\in (0,\pi/2)$, where $G$ is given as above.
\end{thm}

We list some theorems of G. Link in \cite{1}, where the group $G$ does not necessarily satisfy the regular growth condition. They involve the so-called $(b,\theta)$-densities, which are the generalizations of conformal densities in \cite{Albuquerque:1999}, \cite{Dal'Bo_Kim:2008}, \cite{Marc:1993}, \cite{3}, \cite{5}. The analogs of higher product of rank one Hadamard spaces is referred to \cite{15}.
\begin{thm} [Theorem B, \cite{1}]
If $\theta\in (0,\pi/2)$ is such that $\delta_\theta(G)>0$, then there exists a $(b,\theta)$-density for some parameters $b=(b_1,b_2)\in \mathbb{R}^2$.
\end{thm}

\begin{thm} [Theorem C, \cite{1}]
If  a $G$-invariant $(b,\theta)$-density exists for some $\theta\in(0,\pi/2)$, then $\delta_{\theta}\leq b_1cos\theta+b_2sin\theta,$ where $b=(b_1,b_2)\in \mathbb{R}^2.$
\end{thm}

\begin{thm} [Theorem E, \cite{1}]
If $\nu$ is a $(b,\theta)$-density for $\theta\in (0,\pi/2)$ with $\delta_\theta>0$, then a radial limit point is not a point mass for $\nu.$
\end{thm}

Theorem \ref{pos exp} removes completely the assumption that $\delta_\theta>0$ for some $\theta\in (0,\pi/2)$ of G.Link appeared in the above $(b,\theta)$-density existence theorems. We establish them in the following.

\begin{cor}\label{e4}
For any $\theta\in (0,\pi/2)$, there exists a $(b,\theta)$-density with some parameters $b=(b_1,b_2)\in \mathbb{R}^2$.
\end{cor}

\begin{cor}\label{e5}
There exists a $(b,\theta)$-density for any $\theta\in (0,\pi/2)$ with $\delta_{\theta}\leq b_1cos\theta+b_2sin\theta,$ where $b=(b_1,b_2)\in \mathbb{R}^2.$
\end{cor}

\begin{cor}\label{e6}
For any $\theta\in (0,\pi/2)$ and a $(b,\theta)$-density $\nu$, a radial limit point is never an atomic mass for $\nu.$
\end{cor}

In \cite{1}, G. Link proposed a question in the following.
\begin{question}\label{link}
For the unique $\theta_\star\in[0,\pi/2]$ such that $\delta_{\theta_\star}=\max\limits_{\theta\in[0,\pi/2]}\delta_\theta$, does it hold that $\theta_\star\in(0,\pi/2)$ not necessarily with the regular growth condition?
\end{question}
This question is related to the existence of the $(b,0)$-densities and $(b,\pi/2)$-densities. In fact, the Proposition 6.6 in (\cite{1}) provides $b=(b_1,b_2)\in \mathbb{R}^2$ for $\theta=0,\pi/2$. However, in general we do not have $b_2=0$ if $\theta=0$, or $b_1=0$ if $\theta=\pi/2$. But it seems that the above $\theta_\star=0$ (resp.$\theta_\star=\pi/2$) implies the existence of the $(b,0)$-densities (resp. $(b,\pi/2)$-densities). In the end, we will show this is really true under some weak conditions (see Corollary \ref{0density}).

We also give a negative answer to Question \ref{link} by constructing an example such that $\theta_\star=\pi/2$. But the above conclusion holds under a quite mild condition, the precise statement is the following.

\begin{thm}\label{ques 1}
Assume that $(X,d)=(X_1,d_1)\times (X_2,d_2)$ is a product of Hadamard spaces on which a group $G\subset Is(X_1)\times Is(X_2)$ acts properly, $d$ is the standard product metric. If $G$ contains a pair of isometries which projects a pair of independent rank one elements in each factor and satisfies $\delta_\gamma>0$ for $\gamma\in \boldsymbol{(0,\pi/2)}$,
with the following condition:
\begin{align}\label{growth small 1}
\frac{\delta_{\pi/2}}{\delta_\theta}< \frac{1}{sin\theta}\ and\ \frac{\delta_{0}}{\delta_{\pi/2-\beta}}< \frac{1}{sin\beta},
\end{align}
for some $\theta,\beta\in (0,\pi/2)$.

Then there exists some $\theta_\star\in (0,\pi/2)$ such that $\delta_{\theta_\star}=\max\limits_{\theta\in[0,\pi/2]}\delta_\theta$.

\end{thm}
Moreover, the upper bound in (\ref{growth small 1}) is sharp.
An immediate consequence is:
\begin{cor}
Assume that $(X,d)=(X_1,d_1)\times (X_2,d_2)$ is a product of Hadamard spaces on which a group $G\subset Is(X_1)\times Is(X_2)$ acts properly, $d$ is the standard product metric. If $G$ contains a pair of isometries which projects a pair of independent rank one elements in each factor and satisfies $\delta_\theta>0$ for $\theta\in \boldsymbol{[0,\pi/2]}$.

Then there exists some $\theta_\star\in (0,\pi/2)$ such that $\delta_{\theta_\star}=\max\limits_{\theta\in[0,\pi/2]}\delta_\theta$.
\end{cor}

Finally, we give an interesting theorem which extend the $(b,\theta)$-densities.
\begin{thm}\label{0density}
Assume that $(X,d)=(X_1,d_1)\times (X_2,d_2)$ is a product of Hadamard spaces on which a group $G\subset Is(X_1)\times Is(X_2)$ acts properly, $d$ is the standard product metric.

If $G$ contains a pair of isometries which projects a pair of independent rank one elements in each factor and satisfies $\delta_\theta>0$ for $\theta\in (0,\pi/2)$ and $\delta_{0}=\max\limits_{\theta\in[0,\pi/2]}\delta_\theta$ (resp.$\delta_{\pi/2}=\max\limits_{\theta\in[0,\pi/2]}\delta_\theta$).

Then there exist the $(b,0)$-densities (resp. $(b,\pi/2)$-densities).
\end{thm}
The interested reader is referred to see details in the last section.

\section{Preliminaries}

We say a metric space is \emph{proper} if every ball of finite radius is compact. For a proper Hadamard space $Y$, one defines  equivalent classes of geodesic rays by bounded Gromov-Hausdorff distance on account of giving the geometry boundary $\partial Y$ with cone topology. The point of this geometry boundary is a class of a geodesic ray and it is well-known that $\hat Y=Y\cup \partial Y$, $\partial Y$ are compact and $Y$ is open in $\hat Y$. The isometric action of $G\curvearrowright X$ can be extended to a homeomorphic action of $G\curvearrowright \partial Y$.

For a group $G$ acting properly on a Hadamard space, an isometry $g\in G$ is called \emph{rank one} if it acts on a geodesic axis as a translation and the geodesic does not bound a half plane. Note that such $g$ has two fixed points as the endpoints of its axis in $\partial X$.

Given a set $S$, we let $|S|$ be the cardinality. Set
$$A(o,n)=\{g\in G:n-1\leq d(o,go)<n\}.$$

Let $$\delta=\delta(G):=\limsup\limits_{n\rightarrow \infty}\frac{log|A(o,n)|}{n}$$.

We say that $\delta$ is the growth rate of the action $G\curvearrowright X$. $\delta$ is simply called the growth rate of the group $G$ if the action is clear in the context.
Note that $\delta=-\infty$ if $\limsup\limits_{n\rightarrow \infty}|A(o,n)|=0$, and $\delta\geq 0$ if $\limsup\limits_{n\rightarrow \infty}|A(o,n)|\geq 1$.
One easily shows that $\delta\geq 0$ if and only if $\limsup\limits_{n\rightarrow \infty}|A(o,n)|\geq 1$.


In order to characterize the behavior of group $G$ on the space $X$ with slope $\theta\in [0,\pi/2]$, one can introduce
 $$A_{\epsilon,\theta}(o,n):=\{g\in G:n-1\leq d(o,go)< n,|\frac{d_2(o_2,g_2o_2)}{d_1(o_1,g_1o_1)}-tan\theta|\leq \epsilon\},$$
where $\epsilon>0$, $n\in \mathbb{N}_{\geq 1}$. The $\epsilon$-close growth rate of $G$ with slope $\theta$ is defined as
$$\delta_{\epsilon,\theta}=\limsup\limits_{n\rightarrow \infty}\frac{log|A_{\epsilon,\theta}(o,n)|}{n},$$

Next we denote by $\delta_\theta=\liminf\limits_{\epsilon\rightarrow 0}\delta_{\epsilon,\theta}$ the growth rate of $G$ with slope $\theta$.

Similarly, we have $\delta_\theta\geq 0$ if and only if $\limsup\limits_{n\rightarrow \infty}|A_{\epsilon,\theta}(o,n)|\geq 1$ for any $\epsilon>0.$

\begin{dfn}
We say that the action $G\curvearrowright X$ satisfies \emph{the regular growth condition for $[\alpha_1,\alpha_2]$} if $\delta_\theta\geq 0$ for any $\theta\in [\alpha_1,\alpha_2]\subset [0,\pi/2]$,
or equivalently $\limsup\limits_{n\rightarrow \infty}|A_{\epsilon,\theta}(o,n)|\geq 1$ for any $\epsilon>0$ and $\theta\in [\alpha_1,\alpha_2].$

We specify the regular growth condition for $\alpha_1=0,\alpha_2=\pi/2$.
\end{dfn}

G. Link defined a so-called $G$-invariant $(b,\theta)$-density for the product of Hadamard spaces $X$.
\begin{dfn} (Definition 1.2 \cite{1})
Denote by $\mathcal{M}^+(\partial X)$ the cone of positive Borel measures on $\partial X$. Let $\theta\in [0,\pi/2]$ and $b=(b_1,b_2)\in \mathbb{R}^2$. A $G$-invariant $(b,\theta)$-density is a continuous map:
\begin{align*}
\nu: &X\longrightarrow \mathcal{M}^+(\partial X)
\\&x\longmapsto \nu_x
\end{align*}
such that the following statements hold for any $x\in X$ :

\begin{enumerate}
\item $L_G\cap\partial X_\theta\supset supp(\nu_x)\neq \emptyset$
\item $g\nu_x=\nu_{gx}$ for any $g\in G$,
\item
\begin{flushleft}
if $\theta\in (0,\pi/2)$, then $\frac{d\nu_x}{d\nu_o}(\alpha)=exp\{b_1B_{\alpha_1}(o_1,x_1)+b_2B_{\alpha_2}(o_2,x_2)\}$ \\
for any $\alpha=(\alpha_1,\alpha_2,\theta)\in supp(\nu_o)$;\\

if $\theta=0$, then $ b_2=0$ and $\frac{d\nu_x}{d\nu_o}(\alpha)=exp\{b_1B_{\alpha_1}(o_1,x_1)\}$ \\
for any $\alpha=\alpha_1\in supp(\nu_0)$;\\

if $\theta=\pi/2$, then $b_1=0$ and $\frac{d\nu_x}{d\nu_o}(\alpha)=exp\{b_2B_{\alpha_2}(o_2,x_2)\}$ \\
for any $\alpha=\alpha_2\in supp(\nu_0)$;\\

where $B_{\alpha_i}(o_i,x_i)$ is the Busemann function with respect to the geodesic $\alpha_i$ in $X_i$ for $i=1,2$.\\
\end{flushleft}
\end{enumerate}
\end{dfn}

We give the definition of radical limit points for $G$ acting on the product of Hadamard spaces $X$, which is different from that for a group acting on a proper hyperbolic space.
\begin{dfn} (Definition 1.3 \cite{1})
We say a point $\eta\in \partial X$ is a radical limit point of the action $G\curvearrowright X$ that if there is a sequence $(g_n)_n=((g_{n,1},g_{n,2}))_n\subset G$ such that $g_no$ converges to $\eta$ with the following :

Let $i=1,2$. If $\eta=(\eta_1,\eta_2,\theta)\in \partial X_{reg}$, $g_{n,i}o_i$ is in a bounded neighbourhood of one geodesic ray in the class of $\eta_i$.
If $\eta=\eta_i\in \partial X_{sing}$, $g_{n,i}o_i$ is in a bounded neighbourhood of one geodesic ray in the class of $\eta_i$.
\end{dfn}

From now on we let a group $G\subset Is(X_1)\times Is(X_2)$ act properly and isometrically on a product of proper Hadamard spaces $(X,d)=(X_1,d_1)\times (X_2,d_2)$ with a fixed base point $o=(o_1,o_2)$, and $G$ contains $g=(g_1,g_2)$ and $h=(h_1,h_2)$ satisfying that $g_1,h_1$ and $g_2,h_2$ are independent rank one isometries in $Is(X_1)$ and $Is(X_2)$, respectively.

\section{$\delta_{\theta}$ is continuous in the slope $\theta$}
First we introduce a quite useful theorem due to G. Link.

\begin{lemma} [Theorem E,\cite{2}]\label{con}
The homogeneous function $\Psi_G:\mathbb{R}_{\geq 0}^2\longrightarrow \mathbb{R}$ given by $x\longmapsto \|x\|\delta_{\theta(x)}$, is concave, where ${\theta(x)}=arctan\frac{x_2}{x_1},$ for any $x\neq 0$ and $G$ is given as before not necessarily with the regular growth condition.
\end{lemma}

Note that if $G$ is given as before satisfying the regular growth condition, then $\delta_{\theta(x)}\geq 0$.

We give an example which does not satisfy the regular growth condition.
\begin{example}\label{example1}
Let $X$ be the Cayley graph of $F_2\times F_2$ with respect to the standard generators $(a_1,1),(a_2,1),(1,b_1),(1,b_2)$, $G$ be the subgroup of
$$Is(\emph{{Cay}}(F_2,\{a_1,a_2\}))\times Is(\emph{{Cay}}(F_2,\{b_1,b_2\}))$$
generated by $(a_1,b_1),(a_2,b_2)$.

It is easy to see $\theta(g)=\pi/4$ for any $g\in G$. Then
\begin{align*}
&\delta_\theta=-\infty\ \ for\ \theta\in [0,\pi/2]-\{\pi/4\},
\\&\delta_\theta\geq 0\qquad \qquad \ for\ \theta=\pi/4.
\end{align*}
\end{example}

The action $G\curvearrowright X$ does not satisfy the regular growth condition in Example \ref{example1}.

\begin{rmk}
The regular growth condition is used to exclude the singular case $\delta_\theta=-\infty$ for $\theta\in[0,\pi/2]$. Hence it makes sense to discuss the continuity of $\delta_\theta$ in $\theta\in [0,\pi/2].$
\end{rmk}

In the following, we plan to elaborate Theorem \ref{main theorem1}.


\begin{proof}[Proof of Theorem \ref{main theorem1}]
Thanks to the regular growth condition, we have
\begin{align}\label{growth}
\delta\geq\delta_\theta\geq 0,
\end{align}
for any $\theta\in[0,\pi/2]$.

G. Link showed $\delta_{\theta}$ is upper semi-continuous in the slope $\theta$, i.e $\limsup\limits_{\theta_j\rightarrow \theta}\delta_{\theta_j}\leq \delta_{\theta}$ for any sequence $(\theta_j)_j\subset [0,\pi/2]$.

Therefore, it suffices to show the lower semi-continuity of $\delta_{\theta}$.

We argue by contradiction. Otherwise, there are constants $\delta_1,\delta_2$ and a sequence $(\theta_j)_j\subset [0,\pi/2]$, such that $\delta_{\theta_j}<\delta_1<\delta_2<\delta_{\theta}$ and $\lim\limits_{j\rightarrow \infty}\theta_j=\theta$.

Case 1: $\theta\in (0,\pi/2]$. Then there is a subsequence $(\theta_{k_j})_j$ such that
(a) $\theta_{k_j}<\theta$, or (b) $\theta_{k_j}>\theta$.

Subcase (a). We may assume $0\leq\beta<\theta_j<\theta\leq \pi/2$.
Set $H_{\gamma}=(cos\gamma,sin\gamma)$ for any $\gamma\in [0,\pi/2]$. Note that

\begin{align}\label{3}
\Psi(H_{\gamma})=\delta_{\gamma}.
\end{align}

For $t\in [0,1]$, set
$$\tau(t)=arctan\frac{tsin\theta+(1-t)sin\beta}{tcos\theta+(1-t)cos\beta}.$$

A direct computation gives
$$tan\tau(t)=\frac{tsin\theta+(1-t)sin\beta}{tcos\theta+(1-t)cos\beta}.$$
Taking the derivative of $t$, we have
\begin{align*}
&\quad\quad\quad\qquad\qquad\qquad(tan^2\tau(t)+1)\tau'(t)
\\&
=\frac{(sin\theta-sin\beta)(tcos\theta+(1-t)cos\beta)+(cos\beta-cos\theta)(tsin\theta+(1-t)sin\beta)}{(tcos\theta+(1-t)cos\beta)^2},
\end{align*}
for all $t\in[0,1]$.

Furthermore, we get
\begin{align}\label{4}
\tau'(t)=\frac{(sin\theta-sin\beta)(tcos\theta+(1-t)cos\beta)+(cos\beta-cos\theta)(tsin\theta+(1-t)sin\beta)}{(tcos\theta+(1-t)cos\beta)^2+(tsin\theta+(1-t)sin\beta)^2}.
\end{align}

Nevertheless, one deduces that
\begin{align*}
&(tcos\theta+(1-t)cos\beta)^2+(tsin\theta+(1-t)sin\beta)^2\\&\geq t^2+(1-t)^2\\&\geq 2\times (\frac{t+1-t}{2})^2=1/2.
\end{align*}
Then one of the terms
$$(tcos\theta+(1-t)cos\beta)(\geq 0)$$
and
$$(tsin\theta+(1-t)sin\beta)(\geq 0)$$ is positive.

Hence we obtain $\tau'(t)>0$ for all $t\in [0,1]$. $\tau(t)$ depends continuously, and is increasing on $t$. Note that $\tau(0)=\beta,\tau(1)=\theta.$
Therefore there is a unique $t_j\in [0,1]$ satisfying
\begin{align}\label{1}
\tau(t_j)=\theta_j
\end{align}
and
\begin{align}\label{angle app}
\theta_j\rightarrow \theta\ \Leftrightarrow\ t_j\rightarrow 1\Leftrightarrow\ j\rightarrow \infty.
\end{align}

On the other hand, by definition, the equality (\ref{3}), and Lemma \ref{con} we have
\begin{align}
\delta_{\theta_j}&\geq \|tH_\theta+(1-t)H_0\|\Psi_G(H_{\tau(t_j)})\label{convex}
\\&= \Psi_G(t_jH_\theta+(1-t_j)H_0)\nonumber
\\&\geq t_j\delta_\theta+(1-t_j)\delta_0.\nonumber
\end{align}
It is clear that the right hand side tends to $\delta_{\theta}$ as $t_j\rightarrow 1$.

By (\ref{angle app}) and (\ref{convex}), we take sufficiently large $j$ such that $\delta_{\theta_j}\geq \delta_2$, which is a contradiction.

Subcase (b). We may assume $\theta_j>\theta$ and set
\begin{align*}
\tau(t)=arctan\frac{tsin\theta+(1-t)sin(\pi/2)}{tcos\theta+(1-t)cos(\pi/2)},\quad for\ t\in [0,1].
\end{align*}
A similar argument yields a contradiction.

If $\theta=0,$ it is similar to adapt this approach to show the contradiction.

Thus we complete the proof.
\end{proof}

Note that the proof of Theorem \ref{main theorem1} works by replacing $[0,\pi/2]$ with any $[\alpha_1,\alpha_2]\subset [0,\pi/2]$. A straightforward consequence is:
\begin{cor}\label{main cor}
The growth rate $\delta_{\theta}$ of $G$ with slope $\theta$ is continuous in $\theta\in [\alpha_1,\alpha_2]$, where $G$ satisfies the regular growth condition for $[\alpha_1,\alpha_2]\subset [0,\pi/2]$.

\end{cor}

\section{some applications}

In this section, we proceed to show Theorem \ref{pos exp}.

\begin{proof}
Since the images of natural projections
\begin{align*}
&p_1:G\subset Is(X_1)\times Is(X_2)\longrightarrow Is(X_1)
\end{align*}
and
\begin{align*}
p_2:G\subset Is(X_1)\times Is(X_2)\longrightarrow Is(X_2)
\end{align*}
both contain a pair of independent rank one elements. Namely, $g_1,h_1$ and $g_2,h_2$ are independent rank one elements in $p_1(G)$ and $p_2(G)$, respectively.

Note that a pair of independent rank one elements acting on a proper $CAT(0)$ space is a pair of independent contracting elements (see \cite{12}, Theorem 5.4). A well-known fact is that for a proper geodesic metric $X$ on which a group $G\subset Is(X)$ acting properly with two independent contracting elements $g,h\in G$, the subgroup of $G$ generated by $g^N,h^N$ is isomorphic to a free group of rank two for sufficiently large integer $N$ (see \cite{11}, Proposition 2.7).

Therefore $\delta_{\theta_\star}=\delta(G)>0$.
For any $\theta_\star\neq\theta\in (0,\pi/2)$, one of the cases holds: (i) $0<\theta<\theta_\star$,
or (ii) $\theta_\star<\theta<\pi/2$.

In the first case (i), set
$$\tau(t)=arctan\frac{tsin\theta_\star+(1-t)sin0}{tcos\theta_\star+(1-t)cos0},$$
for $t\in [0,1]$.

By the proof of main theorem and $\tau(0)=0,\tau(1)=\theta_\star$, there exists a unique $t_\star\in (0,1)$ such that
$$\tau(t_\star)=\theta\in (0,\theta_\star).$$

Similar to (\ref{convex}), we get
$$\delta_\theta\geq \psi(t_\star H_{\theta_\star}+(1-t_\star)H_0)\geq t_\star\delta_{\theta_\star}+(1-t_\star)\delta_0\geq t_\star\delta_{\theta_\star}>0.$$

In the second case (ii), set
$$\tau(t)=arctan\frac{tsin\pi/2+(1-t)sin\theta_\star}{tcos\pi/2+(1-t)cos\theta_\star},$$
for $t\in [0,1]$. By a similar argument, we show $\tau(t_\star)=\theta$ for the unique $t_\star\in (0,1)$ and $\delta_\theta\geq (1-t_\star)\delta_{\theta_\star}>0$.
\end{proof}

\begin{rmk}
Although G. Link showed there is a unique $\theta_\star\in [0,\pi/2]$ to attain the maximal value of $\delta_\theta$ which is positive, but she did not show $\theta_\star\in (0,\pi/2)$. Thus the assumption on $\delta_\theta>0$ for $\theta\in (0,\pi/2)$ cannot be removed.
\end{rmk}
Now the previous Corollary \ref{e4}, Corollary \ref{e5} and Corollary \ref{e6} are straightforward by Theorem \ref{pos exp}.

In the end of this section, we establish an example.
\begin{example}\label{example2}
Given two free groups of rank 2, $H_1=<a_1,a_2|->,H_2=<b_1,b_2|->$. Let $X=\emph{{Cay}}(H_1,\{a_1,a_2\})\times \emph{{Cay}}(H_2,\{b_1,b_2\})$, and $G$ be the subgroup of $$Is(\emph{{Cay}}(F_2,\{a_1,a_2\}))\times Is(\emph{{Cay}}(F_2,\{b_1,b_2\}))$$
generated by $(a_1,b_1),(a_2,b_2^2)$. Fix $o=(1,1)\in X$ and denote by $x=(a_1,b_1),y=(a_2,b_2^2)$, then any nontrivial element $g\in G$ can be expressed by a word over an alphabet $\{x,x^{-1},y,y^{-1}\}$.

We may assume $g$ contains $n$ elements in $\{x,x^{-1}\}$ and $m$ elements in $\{y,y^{-1}\}$, hence we have
$$tan\theta(g)=\frac{n+2m}{n+m}=\frac{1+2m/n}{1+m/n}\in [0,\infty).$$

It is not hard to show, the closure of the set of $\theta(g)$ in $[0,\infty)$ is $[\pi/4,arctan2]$ for any $g\in G$. Therefore, we deduce
\begin{align*}
  &\delta_\theta=-\infty,\ for\ \theta\in [0,\pi/4)\cup (arctan2,\pi/2];
\\&\delta_\theta\geq 0,\qquad\ for\ \theta\in [\pi/4,arctan2].
\end{align*}
Using our Corollary \ref{main cor}, one can obtain $\delta_\theta$ is continuous in $\theta\in [\pi/4,arctan2]$.
\end{example}

\section{Results on Question \ref{link}}
At present, we first give an example such that $\delta_{\pi/2}=\max\limits_{\theta\in[0,\pi/2]}\delta_\theta$. This implies a negative answer to Question \ref{link} in general.
\begin{example}\label{counterxample}
Given the product of Hadamard spaces $X=(Cay(F_2,T),d_1)\times (Cay(F_{N},S),d_2)$, where $Cay(F_2,T)$ (resp. $Cay(F_{N},S)$) is the Cayley graph of the free group $F_2$ (resp. $F_{N}$) of rank two (resp. $N$) with the word metric $d_1$ (resp. $d_2$) induced by the standard generating set $T$ (resp. $S$). We may assume $T=\{a_1,a_2\}$ and $S=\{b_1,b_2,\cdots,b_N\}$. Let $G$ be the subgroup in $F_2\times F_N$ generated by $g_1=(a_1,b_1),g_2=(a_2,b_2),g_3=(1,b_3),g_4=(1,b_4),\cdots,g_N=(1,b_N)$. Fix the base point $o=(1,1)$.

We will show for any $\theta\in [\pi/4,\pi/2]$,
\begin{align}\label{growth 1}
\delta_\theta=log(2(N-2)-1)\cdot(sin\theta)-(log(2(N-2)-1)-log3)\cdot cos\theta.
\end{align}
As a consequence, for $N\geq 4$ we have
$$\delta_{\pi/2}=\max\limits_{\theta\in[0,\pi/2]}\delta_\theta.$$
\end{example}

\begin{proof}
For any $1\neq g\in G$, it is clear that $g$ can be expressed uniquely in terms of $g_1,g_2,\cdots,g_N$ by definition and the fact that $F_N$ is free.
We may assume $g=(g_{(1)},g_{(2)})=g_{i_1}^{k_1} g_{i_2}^{k_2}\cdots g_{i_s}^{k_s}$, where $i_1,i_2,\cdots,i_s\in \{1,2,\cdots,N\}$ and $k_1,\cdots,k_s$ are integers.
For convenience, one can assume $i_1,i_2,\cdots,i_{t}\in \{1,2\}$ and $i_{t+1},i_{t+2},\cdots,i_{s}\in \{3,4,\cdots,N\}$ with $1\leq t\leq N$.
Then we have
\begin{align}\label{tang}
tan\theta(g)=\frac{d_2(1,g_{(2)})}{d_1(1,g_{(1)})}=\frac{\sum\limits_{j=1}^{N}|k_j|}{\sum\limits_{j=1}^t|k_j|}.
\end{align}
It is apparent from (\ref{tang}) that $\theta(g)\in [\pi/4,\pi/2]$, $\delta_\theta\geq 0$ for $\theta\in [\pi/4,\pi/2]$ and $\delta_\theta=-\infty$ for $\theta\in [0,\pi/4)$.
For any $g\in A_{\epsilon,\theta}(o,n)$ with $\epsilon>0$, $\theta\in [\pi/4,\pi/2]$,
it is not hard to show there exists $c(\epsilon,n)>0$ with $\lim\limits_{\epsilon\rightarrow 0}\lim\limits_{n\rightarrow \infty}c(\epsilon,n)/n=0$ such that
\begin{align}\label{g1g2}
|d_2(1,g_{(2)})-nsin\theta|\leq c(\epsilon,n),\ |d_1(1,g_{(1)})-ncos\theta|\leq c(\epsilon,n).
\end{align}
Recall that $g$ can be expressed uniquely in terms of $g_1,g_2,\cdots,g_N$. If one assumes $n_1$ (resp. $n_2$) to be the frequency of occurrence of $\{g_1^{\pm1},g_2^{\pm1}\}$ (resp. $g_3^{\pm1}=(1,b_3)^{\pm1},g_4^{\pm1}=(1,b_4)^{\pm1},\cdots,g_N^{\pm1}=(1,b_N)^{\pm1}$) for the expression of $g$, via (\ref{tang}) we get
\begin{align}\label{n12}
d_2(1,g_{(2)})=n_1+n_2, d_1(1,g_{(1)})=n_1.
\end{align}
Using (\ref{g1g2}), (\ref{n12}), we have
\begin{align}\label{estn12}
|n_2-n(sin\theta-cos\theta)|\leq 2c(\epsilon,n),\ |n_1-ncos\theta|\leq c(\epsilon,n).
\end{align}
On the other hand, it is well-known that the finitely generated free groups have purely exponential growth (see \cite{11}, \cite{13}). Note that for any positive integer $m$, we have $\delta(F_m)=log(2m-1)$ with respect to the standard generating set. Hence via (\ref{estn12}), we have
\begin{align}
& \frac{1}{c(F_{N-2})}exp\{log(2(N-2)-1)\cdot (n(sin\theta-cos\theta)-2c(\epsilon,n))\}\nonumber
\\&\cdot\frac{1}{c(F_2)}exp\{log3\cdot(ncos\theta-c(\epsilon,n))\}\nonumber
\\&\leq A_{\epsilon,\theta}(o,n)\label{f2fn}
\\&\leq c(F_{N-2})exp\{log(2(N-2)-1)\cdot (n(sin\theta-cos\theta)+2c(\epsilon,n))\}\nonumber
\\&\cdot c(F_2)exp\{log3\cdot(ncos\theta+c(\epsilon,n))\}.\nonumber
\end{align}
where $c(F_2)$,(resp. $c(f_N)$) are constants depending only on $F_2$ (resp. $F_N$).

Thus this yields the following for $N\geq 4$
\begin{align}
\delta_\theta=&\liminf\limits_{\epsilon\rightarrow 0}\limsup\limits_{n\rightarrow \infty}\frac{log|A_{\epsilon,\theta}(o,n)|}{n}\nonumber
\\&=log(2(N-2)-1)\cdot (sin\theta-cos\theta)+log3\cdot cos\theta\nonumber
\\&=log(2(N-2)-1)\cdot (sin\theta)-(log(2(N-2)-1)-log3)\cdot cos\theta\nonumber
\\&\leq log(2(N-2)-1)=\delta_{\pi/2}.\nonumber
\end{align}

\end{proof}

In the next, we will present Theorem \ref{ques 1} which provides an affirmative answer to Question \ref{link} under a quite mild condition.
\begin{thm}\label{ques}
Assume that $(X,d)=(X_1,d_1)\times (X_2,d_2)$ is a product of Hadamard spaces on which a group $G\subset Is(X_1)\times Is(X_2)$ acts properly, $d$ is the standard product metric. If $G$ contains a pair of isometries which projects a pair of independent rank one elements in each factor and satisfies $\delta_\gamma>0$ for \textbf{$\gamma\in (0,\pi/2)$},
with the following condition:
\begin{align}\label{growth small}
\frac{\delta_{\pi/2}}{\delta_\theta}< \frac{1}{sin\theta}\ and\ \frac{\delta_{0}}{\delta_{\pi/2-\beta}}< \frac{1}{sin\beta},
\end{align}
for some $\theta,\beta\in (0,\pi/2)$.

Then there exists some $\theta_\star\in (0,\pi/2)$ such that $\delta_{\theta_\star}=\max\limits_{\theta\in[0,\pi/2]}\delta_\theta$.

\end{thm}
\begin{proof}
Since $\delta_0\geq \limsup\limits_{\theta_j\rightarrow 0}\delta_{\theta_j}\geq 0$ and $\delta_{\pi/2}\geq \limsup\limits_{\theta_i\rightarrow \pi/2}\delta_{\theta_i}\geq 0$,
where $\theta_j,\theta_i\in (0,\pi/2)$. Then $\delta_{\theta}$ are nonnegative in $[0,\pi/2]$.

Recall that for any $0\neq x=(x_1,x_2)\in \mathbb{R}^2_{\geq 0}$, $\theta(x)=arctan\frac{x_1}{x_1}$ and set $\theta(0)=0$. The key ingredient of the proof is the observation:

The following holds for any $x,y\in \mathbb{R}^2_{\geq 0}$ and any $t\in [0,1]$,
\begin{align}\label{identity}
||tx+(1-t)y||{sin(\theta(tx+(1-t)y))}=t||x||sin(\theta(x))+(1-t)||y||sin(\theta(y)).
\end{align}
\begin{proof}

It is trivial for $x=0$ or $y=0$. Hence we may assume $||x||$, $||y||>0$. One can identify any vector in $\mathbb{R}^2_{\geq 0}$ with a point in the plane, so we let $O$ denote $0$ and $P\in \mathbb{R}^2_{\geq 0}$ (resp. $Q$) denote $x$ (resp. $y$). Suppose $\hat d$ is the standard Euclidean metric of the plane.

We may assume $0\leq \theta(x)<\theta(y)\leq \pi/2$. Denote by $R$ the vector ${tx+(1-t)y}$. One can extend the segment $[O,R]$ to the point $T$ such that
$$\frac{\hat d(O,R)}{\hat d(T,R)}=\frac{\hat d(P,R)}{\hat d(Q,R)}.$$
Note that the triangle $\Delta TQR$ is similar to the triangle $\Delta OPR$. Thus we have
\begin{align}
&\frac{\hat d(O,R)}{\hat d(T,R)}=\frac{\hat d(P,R)}{\hat d(Q,R)}=\frac{\hat d(O,P)}{\hat d(T,Q)}=\frac{1-t}{t}.\label{simi1}
\\&\hat d(O,T)=\hat d(O,R)+\hat d(R,T)\label{simi2}
\\&=||tx+(1-t)y||+\frac{t}{1-t}||tx+(1-t)y||\nonumber
\\&=\frac{||tx+(1-t)y||}{1-t}.\nonumber
\end{align}

If we set $\alpha=\angle TOP$, $\beta=\angle TOQ$, we have
\begin{align}\label{angle sum}
\alpha+\beta=\theta(y)-\theta(x).
\end{align}
Using the Law of Sines for the triangle $\Delta TQO$, (\ref{simi1}) and (\ref{simi2}), we get
\begin{align}
&\frac{sin\alpha}{sin(\pi-\alpha-\beta)}=\frac{\hat d(O,Q)}{\hat d(O,T)}=\frac{(1-t)||y||}{||tx+(1-t)y||},\label{angle rela1}
\\&\frac{sin\beta}{sin(\pi-\alpha-\beta)}=\frac{\hat d(T,Q)}{\hat d(O,T)}=\frac{t||x||}{||tx+(1-t)y||}.\label{angle rela2}
\end{align}
Thus we have
\begin{align}
&||tx+(1-t)y||{sin(\theta(tx+(1-t)y))}\nonumber
\\&=t||x||sin(\theta(x))+(1-t)||y||sin(\theta(y))\nonumber
\\&\Leftrightarrow \frac{sin\beta sin\theta(x)+sin\alpha sin\theta(y)}{sin(\theta(y)-sin\theta(x))}=sin(\alpha+\theta(x))\label{p2}
\\&\Leftrightarrow sin\beta sin\theta(x)+sin\alpha sin\theta(y)\label{p3}
\\&=(sin\alpha cos\theta(x)+cos\alpha sin\theta(x))sin(\theta(y)-sin\theta(x))\nonumber
\\&\Leftrightarrow sin\alpha(sin((\theta(y)-\theta(x))+\theta(x))-cos\theta(x)sin(\theta(y)-\theta(x)))\label{p4}
\\&=(cos\alpha sin(\theta(y)-\theta(x))-sin((\theta(y)-\theta(x))-\alpha))sin\theta(x)\nonumber
\\&\Leftrightarrow sin\alpha sin\theta(x)cos(\theta(y)-\theta(x))\label{p5}
\\&=cos(\theta(y)-\theta(x))sin\alpha sin\theta(x),\nonumber
\end{align}
where (\ref{p2}) we used (\ref{angle rela1}) and (\ref{angle rela2}), (\ref{p4}) we used (\ref{angle sum}), and (\ref{p3}) (\ref{p5}) we used the basic properties of sine functions.
\end{proof}
Recall that for any $0\neq x,y\in \mathbb{R}^2_{\geq 0}$, we have
\begin{align}\label{conca}
||tx+(1-t)y||\delta_{\theta(tx+(1-t)y)}\geq t||x||\delta_{\theta(x)}+(1-t)||y||\delta_{\theta(y)}.
\end{align}
Take
\begin{align}\label{xx}
\theta(x)=\pi/2,
\end{align}
and fix $0\neq x,y\in \mathbb{R}^2_{\geq 0}$ with $0<\theta(y)<\pi/2$.

Similar to (\ref{4}), one can deduce that $\theta(tx+(1-t)y)$ is smooth on $t\in [0,1]$. Then this yields
\begin{align}
&\lim\limits_{t\rightarrow 1}\frac{||tx+(1-t)y||(1-sin\theta(tx+(1-t)y))}{(1-t)||y||}\nonumber
\\&=\frac{||x||}{||y||}\lim\limits_{t\rightarrow 1}cos\theta(tx+(1-t)y)\theta'(tx+(1-t)y)\label{lhos}
\\&=\frac{||x||}{||y||}cos\theta(x)\theta'(x)\nonumber
\\&=0,\label{ppp3}
\end{align}
where we applied L'Hospital's Rule to get (\ref{lhos}) and applied (\ref{xx}) to get (\ref{ppp3}).

By (\ref{ppp3}) and $\frac{\delta_{\pi/2}}{\delta_\theta}< \frac{1}{sin\theta}$ , for $t\in (0,1)$ such that $t$ is sufficiently close to $1$ we obtain
\begin{align}\label{ppp4}
\frac{(1-t)||y||}{||tx+(1-t)y||(1-sin\theta(tx+(1-t)y))+(1-t)||y||sin\theta(y)}>\frac{\delta_{\pi/2}}{\delta_\theta(y)}.
\end{align}
Combining (\ref{ppp4}), (\ref{xx}) with (\ref{identity}), one can have
\begin{align}
t||x||\delta_{\theta(x)}+(1-t)||y||\delta_{\theta(y)}&=t||x||\delta_{\pi/2}+(1-t)||y||\delta_{\theta(y)}\nonumber
\\&> ||tx+(1-t)y||\delta_{\pi/2}=||tx+(1-t)y||\delta_{\theta(x)}.\label{ppp5}
\end{align}
On the other hand, by (\ref{conca}) and (\ref{ppp5}) we get
\begin{align*}
||tx+(1-t)y||\delta_{\theta(tx+(1-t)y)}>||tx+(1-t)y||\delta_{\pi/2}.
\end{align*}
Therefore we deduce $\delta_{\theta(tx+(1-t)y)}>\delta_{\pi/2}$ for some $t\in (0,1)$ such that $t$ is sufficiently close to $1$.

If one exchange $X_1$ and $X_2$, then for $G\curvearrowright X=X_2\times X_1$ we have $\tilde{\delta}_\beta=\delta_{\pi/2-\beta}$. Similarly, for some $t\in (0,1)$ such that $t$ is sufficiently close to $1$, we have
$$\delta_{\pi/2-\theta(tx+(1-t)y)}=\tilde{\delta}_{\theta(tx+(1-t)y)}>\tilde{\delta}_{\pi/2}=\delta_0.$$
Thus we complete it.
\end{proof}
\begin{rmk}
Recall the example \ref{counterxample}. If one takes $N=4$, then $\delta_{\theta}=sin\theta\cdot \delta_{\pi/2}$ by the equation (\ref{growth 1}). Hence $\frac{\delta_{\pi/2}}{\delta_\theta}=\frac{1}{sin\theta}.$ This implies the upper bound in (\ref{growth small}) for $\frac{\delta_{\pi/2}}{\delta_\theta}$ in Theorem \ref{ques} is sharp.
\end{rmk}

It is time to give a proof of Theorem \ref{0density}.
\begin{proof}
We may assume $\delta_{0}=\max\limits_{\theta\in[0,\pi/2]}\delta_\theta$. Then by Theorem \ref{ques}, the inequalities (\ref{growth small}) fails. Hence, for all $\theta\in (0,\pi/2)$ we have
\begin{align}\label{i1}
\frac{\delta_{\pi/2}}{\delta_\theta}\geq\frac{1}{sin\theta}
\end{align}
or
\begin{align}\label{i2}
\frac{\delta_{0}}{\delta_{\pi/2-\theta}}\geq\frac{1}{sin\theta}.
\end{align}
Recall that $\delta_{0}=\max\limits_{\theta\in[0,\pi/2]}\delta_\theta>0$, if the inequality (\ref{i1}) holds and by Theorem \ref{main theorem1}, we can deduce $\delta_{\pi/2}\geq \frac{\delta_{\theta_1}}{sin\theta_1}>\delta_0$ for some small enough $\theta_1\in (0,\pi/2)$, which is a contradiction.

Thus the inequality (\ref{i2}) holds. By (\ref{i2}) and Theorem \ref{main theorem1}, one can get $0\leq \delta_{\pi/2}\leq \delta_0sin0=0$. Then we obtain $\delta_{\pi/2}=0$.

Claim: one can choose $b_1=\delta_0,b_2=\delta_{\pi/2}$ in the proof of Proposition 6.6 in \cite{1}, which is used to construct some $(b,\theta)$-density.

\begin{proof}[Proof of the Claim]
It suffices to show that $\delta_\theta\geq \delta_0cos\theta+\delta_{\pi/2}sin\theta$. Recall that $H_\theta=(cos\theta,sin\theta)$ for $\theta\in[0,\pi/2]$. It is straightforward to verify that, for $t=\dfrac{sin\theta}{1+cos\theta}$, we have
\begin{align}
&tH_{\pi/2}+(1-t)H_0=(t^2+(1-t)^2)^{1/2}H_\theta,\nonumber
\\&\dfrac{t}{(t^2+(1-t)^2)^{1/2}}=sin\theta,\ \dfrac{1-t}{(t^2+(1-t)^2)^{1/2}}=cos\theta. \label{i3}
\end{align}

By Lemma \ref{con}, we get
\begin{align}\label{i4}
&(t^2+(1-t)^2)^{1/2}\delta_\theta=(t^2+(1-t)^2)^{1/2}\Psi(H_\theta)
\\=&\Psi(tH_{\pi/2}+(1-t)H_0)\geq t\Psi(H_{\pi/2})+(1-t)\Psi(H_0)=t\delta_{\pi/2}+(1-t)\delta_0.
\end{align}
Using (\ref{i3}), one can obtain $\delta_\theta\geq \delta_0cos\theta+\delta_{\pi/2}sin\theta$.
\end{proof}

So $b_2=\delta_{\pi/2}=0$, thus we complete the proof.

\end{proof}

\begin{acknowledgments}
The author is very grateful to Prof. Wen-yuan Yang and Prof. Jinsong Liu for many helpful discussions and suggestions.
\end{acknowledgments}

\end{document}